\documentclass[11pt,reqno,tbtags]{amsart}
\usepackage{amssymb}
\usepackage{url}
\usepackage[square,numbers]{natbib}
\bibpunct[, ]{[}{]}{;}{n}{,}{;}

\title
{Patterns in random permutations avoiding the pattern 321}

\date{25 September, 2017} 

\author{Svante Janson}
\thanks{Partly supported by the Knut and Alice Wallenberg Foundation}
\address{Department of Mathematics, Uppsala University, PO Box 480,
SE-751~06 Uppsala, Sweden}
\email{svante.janson@math.uu.se}
\newcommand\urladdrx[1]{{\urladdr{\def~{{\tiny$\sim$}}#1}}}
\urladdrx{http://www2.math.uu.se/~svante/}


\numberwithin{equation}{section}

\renewcommand\le{\leqslant}
\renewcommand\ge{\geqslant}

\allowdisplaybreaks

\newtheorem{theorem}{Theorem}[section]
\newtheorem{lemma}[theorem]{Lemma}

\newtheorem{corollary}[theorem]{Corollary}

\theoremstyle{definition}
\newtheorem{example}[theorem]{Example}

\newtheorem{problem}[theorem]{Problem}
\newtheorem{remark}[theorem]{Remark}

\theoremstyle{remark}

\newenvironment{romenumerate}[1][0pt]{
\addtolength{\leftmargini}{#1}\begin{enumerate}
 }{\end{enumerate}}

\newcounter{oldenumi}
{\setcounter{oldenumi}{\value{enumi}}
\begin{romenumerate} \setcounter{enumi}{\value{oldenumi}}}
{\end{romenumerate}}

\newcounter{thmenumerate}
\newenvironment{thmenumerate}
{\setcounter{thmenumerate}{0}%
 \def\item{\par
 \refstepcounter{thmenumerate}\textup{(\roman{thmenumerate})\enspace}}
}
{}

\newcounter{romxenumerate}   

\newcounter{xenumerate}   

\newcommand\pfitemx[1]{\par#1:}
\newcommand\pfitemref[1]{\pfitemx{\ref{#1}}}

\newcommand{\refT}[1]{Theorem~\ref{#1}}
\newcommand{\refC}[1]{Corollary~\ref{#1}}
\newcommand{\refL}[1]{Lemma~\ref{#1}}
\newcommand{\refR}[1]{Remark~\ref{#1}}
\newcommand{\refS}[1]{Section~\ref{#1}}
\newcommand{\refSS}[1]{Subsection~\ref{#1}}

\newcommand{\refE}[1]{Example~\ref{#1}}

\newenvironment{comment}{\setbox0=\vbox\bgroup}{\egroup} 

\begingroup
  \divide\count255 by 60
  \multiply\count255 by -60
  \advance\count255 by \time
  \ifnum \count255 < 10 \xdef\klockan{\the\count1.0\the\count255}
  \else\xdef\klockan{\the\count1.\the\count255}\fi
\endgroup

\newcommand{\sumkn}{\sum_{k=1}^n}

\newcommand\set[1]{\ensuremath{\{#1\}}}

\newcommand\bigpar[1]{\bigl(#1\bigr)}
\newcommand\Bigpar[1]{\Bigl(#1\Bigr)}
\newcommand\biggpar[1]{\biggl(#1\biggr)}

\newcommand\bigsqpar[1]{\bigl[#1\bigr]}

\newcommand\bigabs[1]{\bigl|#1\bigr|}

\def\rompar(#1){\textup(#1\textup)}    
\newcommand\xfrac[2]{#1/#2}

\newcommand\Bigparfrac[2]{\Bigpar{\frac{#1}{#2}}}

\def\xexp(#1){e^{#1}}
\newcommand\ceil[1]{\lceil#1\rceil}
\newcommand\floor[1]{\lfloor#1\rfloor}

\newcommand\ntoo{\ensuremath{{n\to\infty}}}

\newcommand\punkt{.\spacefactor=1000}    
\newcommand\iid{i.i.d\punkt}    
\newcommand\ie{i.e\punkt}
\newcommand\eg{e.g\punkt}

\newcommand{\as}{a.s\punkt}

\newcommand\whp{with high probability}
\newcommand\vhp{with very high probability}

\newcommand{\tend}{\longrightarrow}
\newcommand\dto{\overset{\mathrm{d}}{\tend}}
\newcommand\pto{\overset{\mathrm{p}}{\tend}}

\newcommand\eqd{\overset{\mathrm{d}}{=}}

\newcommand\bbR{\mathbb R}

\newcommand\bbZ{\mathbb Z}

\newcounter{CC}
\newcounter{cc}
\newcommand{\cc}{\stepcounter{cc}\ccx} 
\newcommand{\ccx}{c_{\arabic{cc}}}     

\newcommand\E{\operatorname{\mathbb E{}}}
\renewcommand\P{\operatorname{\mathbb P{}}}
\newcommand\Var{\operatorname{Var}}
\newcommand\Cov{\operatorname{Cov}}

\newcommand\gd{\delta}
\newcommand\gD{\Delta}

\newcommand\gam{\gamma}
\newcommand\gG{\Gamma}
\newcommand\gk{\kappa}
\newcommand\gl{\lambda}

\newcommand\gs{\sigma}

\newcommand\gt{\tau}
\newcommand\eps{\varepsilon}

\renewcommand\phi{\xxx}  

\newcommand\cA{\mathcal A}
\newcommand\cB{\mathcal B}

\newcommand\cD{\mathcal D}
\newcommand\cE{\mathcal E}

\newcommand\cN{\mathcal N}

\newcommand\cP{\mathcal P}
\newcommand\cQ{\mathcal Q}

\newcommand\qw{^{-1}}

\newcommand\qq{^{1/2}}
\newcommand\qqw{^{-1/2}}

\newcommand\qqc{^{3/2}}

\newcommand\intoi{\int_0^1}
\newcommand\intoo{\int_0^\infty}

\newcommand\oi{[0,1]}
\newcommand\ooo{[0,\infty)}

\newcommand\dd{\,\mathrm{d}}

\newcommand\lhs{left-hand side}

\newcommand\fS{\mathfrak S}

\newcommand\fSx{\fS_*}
\newcommand\fSnwww{\fS_n(\www)}

\newcommand\fSmwww{\fS_m(\www)}
\newcommand\fSxwww{\fSx(\www)}
\newcommand\npp[1]{n_\perm{#1}}
\newcommand\vpp[1]{v_\perm{#1}}
\newcommand\wpp[1]{w_\perm{#1}}
\newcommand\Wpp[1]{W_\perm{#1}}
\newcommand\nt{n_\gt}
\newcommand\ns{n_\gs}
\newcommand\perm[1]{\ensuremath{\mathbf{#1}}}
\newcommand\pint{\pinx{\gt}}
\newcommand\pinx[1]{\boldsymbol{\pi}_{{#1},n}}
\newcommand\pinxp[1]{\boldsymbol{\pi}_{\perm{#1},n}}
\newcommand\pinxx[1]{\boldsymbol{\pi}_{#1;n}}
\newcommand\pix{\boldsymbol{\pi}}
\newcommand\pinwww{\pinx{\www}}
\newcommand\www{\perm{321}}
\newcommand\zzz{\perm{132}}

\newcommand\Ep{E_+}
\newcommand\Em{E_-}
\newcommand\Dyck{\mathfrak{D}}
\newcommand\Dyckn{\mathfrak{D}_{2n}}
\newcommand\cp{^{\mathsf{c}}}
\newcommand\bex{\mathbf{e}}

\newcommand\brbr{\mathbf{b}}
\newcommand\bk{\mathbf{k}}
\newcommand\gDX{\bar{\gD}}

\newcommand\MXn{\gDX_n}
\newcommand\PC{`Petrov conditions'}
\newcommand\pig{\pi_\gamma}
\newcommand\bkk{(k_1,\dots,k_m)}
\newcommand\nqq{\bigpar{n\qq}}
\newcommand\Pigso{\Pi_\gs^\circ}
\newcommand\cao{\cA^0}
\newcommand\ellx{x}
\newcommand\wz{z}

\begin{document}

\begin{comment}  
05 Combinatorics 
05A Enumerative combinatorics [For enumeration in graph theory, see 05C30]
05A05 Permutations, words, matrices 

60 Probability theory and stochastic processes
60C Combinatorial probability
60C05 Combinatorial probability

60F Limit theorems [See also 28Dxx, 60B12]
60F05 Central limit and other weak theorems
\end{comment}

\begin{abstract} 
We consider a random permutation drawn from
the set of \www-avoiding permutations of length $n$ and show that
the number of occurrences of another pattern $\sigma$ has a limit distribution,
after scaling by $n^{m+\ell}$ where $m$ is the length of $\sigma$ and $\ell$
is the number of blocks in it.
The limit is not normal, and can be expressed as a functional of a Brownian
excursion. 
\end{abstract}

\maketitle

\section{Introduction}\label{S:intro}

Let $\fS_n$ be the set of permutations of $[n]:=\set{1,\dots,n}$,
and $\fSx:=\bigcup_{n\ge1}\fS_n$.
If $\gs=\gs_1\dotsm\gs_m\in\fS_m$ and $\pi=\pi_1\dotsm\pi_n\in\fS_n$,
then an \emph{occurrence} of $\gs$ in $\pi$ is a sequence
$(i_1,\dots,i_m)$ with
$1\le i_1<\dots<i_m\le n$, such that the subsequence
$\pi_{i_1}\dotsm\pi_{i_m}$ has the same order as
$\gs$, i.e., 
$\pi_{i_j}<\pi_{i_k} \iff \gs_j<\gs_k$ for all $j,k\in [m]$.
We let $\ns(\pi)$ be the number of occurrences of $\gs$ in $\pi$, and note
that
\begin{equation}\label{11}
  \sum_{\gs\in\fS_m} \ns(\pi) = \binom nm,
\end{equation}
for every $\pi\in\fS_n$.
For example, an inversion is an occurrence of \perm{21}, and thus
$\npp{21}(\pi)$ is the number of inversions in $\pi$.

\begin{remark}
  It is often natural to think of an occurence as the subsequence
  $\pi_{i_1}\dotsm\pi_{i_m}$ rather than the corresponding sequence of
  indices $i_1,\dotsm,{i_m}$. However, in the present paper we use 
in formal arguments
the  definition above with indices. 
\end{remark}

We say that $\pi$ \emph{avoids} another permutation 
$\tau$ if $\nt(\pi)=0$; otherwise, $\pi$
\emph{contains} $\tau$. Let 
\begin{equation}\label{fsgt}
\fS_n(\tau):= \set{\pi\in\fS_n:\nt(\pi)=0},
\end{equation}
the set 
of permutations of length $n$ that avoid $\tau$.
We also let $\fSx(\tau):=\bigcup_{n=1}^\infty\fS_n(\tau)$ be the set of
$\tau$-avoiding permutations of arbitrary length. 

The classes $\fSx(\gt)$ of $\tau$-avoiding permutations have been studied
for a long time, see \eg{}
\citet[Exercise 2.2.1-5]{KnuthI}, 
\citet{Simion-Schmidt}, 
\citet{BJS}. 
One classical  problem is to enumerate the sets
$\fS_n(\gt)$, either exactly or asymptotically,
see \citet[Chapters 4--5]{Bona}.
We note here only the fact that for any $\tau$ with $|\tau|=3$,
$\fS_n(\gt)$ has the same size
\begin{equation}
  \label{catalan}
|\fS_n(\gt)|=C_n:=\frac{1}{n+1}\binom{2n}{n},
\end{equation}
the $n$-th Catalan number, see \eg{} 
\cite[Exercises 2.2.1-4,5]{KnuthI}, 
\cite{Simion-Schmidt},
\cite[Exercise 6.19ee,ff]{StanleyII},
\cite[Corollary 4.7]{Bona}.
(The situation for $|\tau|\ge4$ is more complicated.)

The general problem that concerns us is to
take a fixed permutation 
$\gt$
and let $\pint$ be a uniformly random $\gt$-avoiding permutation, \ie, a
uniformly random element of $\fS_n(\gt)$, and then study the 
distribution of the random variable $\ns(\pint)$
for some other fixed permutation $\gs$.
(Only $\gs$ that are $\tau$-avoiding are interesting, 
since otherwise $\ns(\pint)=0$.)
One instance of this problem was studied already by \citet{RobertsonWZ},
who gave a  generating function for $\npp{123}(\pinxp{132})$.
The exact distribution of $\ns(\pint)$ 
for a given $n$ was studied 
numerically
in \cite{SJ287}, where higher
moments and mixed moments
are calculated for small $n$ 
for 
several cases ($\gt= \perm{132}$, \perm{123} and \perm{1234};  
several $\gs$ with $|\gs|=3$).

We are mainly interested in 
asymptotics of the distribution 
of $\ns(\pint)$. and of its moments,
as \ntoo{}, for some fixed $\gt$ and $\gs$. 
The 
case $\gt=\zzz$ and arbitrary $\gs$ were studied in detail in \cite{SJ290}.
In the present paper
we study the case $\gt=\www$.
Together with obvious symmetries, these two cases cover all cases 
where $\tau$ has length $|\tau|=3$. (Note that the cases with
$|\tau|=2$ are trivial.)
The cases with $|\gt|\ge4$ seem much more difficult, and are left as
challenging open problems to the readers.

The 
expectation $\E\ns(\pint)$, or
equivalently, the total number of 
occurences of $\gs$ in all $\tau$-avoiding permutations, has been treated
in a number of papers
for various cases,
beginning with \citet{Bona-abscence,Bona-surprising} (with $\tau=\perm{132}$).
In particular, for
the $\perm{321}$-avoiding permutations studied in the present paper,
\citet{CEF} gave an exact formula for the total number of inversions 
(occurences of \perm{21}),
and \citet{Homberger} gave generating functions for the total number of
occurences of $\gs$ in $\fS_n(\perm{321})$ for all $\gs$ with $|\gs|\le3$
and as a consequence asymptotic formulas 
as \ntoo{}
for these numbers.
(The results in \cite{Homberger} are really stated for $\fS_n(\perm{123})$,
which is equivalent.) 
These results in \cite{CEF} and \cite{Homberger} 
imply (after correcting some typos in \cite{Homberger}), in our notation,
\begin{align}
  \E \npp{21}(\pinxp{321}) 
&\sim \frac{\sqrt\pi}4 n^{3/2}, \label{e12H}
\\
  \E \npp{231}(\pinxp{321}) 
&=
  \E \npp{312}(\pinxp{321})
\sim \frac14 n^2, \label{e132H}
\\
\E \npp{132}(\pinxp{321})
&=  \E \npp{213}(\pinxp{321})
\sim \frac{\sqrt\pi}{8} n^{5/2}, \label{e231H}
\\
  \E \npp{123}(\pinxp{321}) 
&\sim\binom n3 \sim \frac16 n^3. \label{e321H}
\end{align}
Moreover,
the equivalence given by \citet{CEF} between $\npp{21}(\pinxp{321})$ and the
number of certain squares under a Catalan path 
implies by standard results 
for the area under the equivalent Dyck paths
that, as \ntoo,
\begin{equation}\label{qj}
  n^{-3/2} \npp{21}(\pinxp{321})\dto 2\qqw \intoi \bex(x)\dd x,
\end{equation}
where the limit random variable is, 
up to a constant factor,
the area under a  Brownian excursion $\bex$
(see \eg{} \cite{SJ201} for many other results on this random area).
See also the related expressions for the distribution of 
$\npp{21}(\pinxp{321})$ in
\citet{ChenMeiWang}.

Apart from \eqref{qj}, we do not know any previous result on asymptotic
distributions of $\ns(\pinxp{321})$ beyond the expectations in
\eqref{e12H}--\eqref{e321H}. 

Our main result is the following, using the notion of blocks defined in
\refSS{SSblocks} below.
The proof is given in \refS{Spf} below, 
and is based on results for \www-avoiding permutations by \citet{HRS-I,HRS-II}.

\begin{theorem}\label{T1}
  Let $\gs$ be a fixed \www-avoiding permutation.
Let $m:=|\gs|$, 
suppose that $\gs$ has $\ell$ blocks of lengths $m_1,\dots,m_\ell$, and let
$w_\gs$ be the positive constant defined in \eqref{wgs}.
Then, as \ntoo,
\begin{equation}\label{t1a}
  \ns(\pinx{\www})/n^{(m+\ell)/2}\dto W_\gs
\end{equation}
for a positive random variable $W_\gs$ that can be represented as
\begin{equation}\label{t1w}
  W_\gs= w_\gs\int_{0<t_1<\dots<t_\ell<1} 
\bex(t_1)^{m_1-1}\dotsm \bex(t_\ell)^{m_\ell-1}\dd t_1\dotsm \dd t_\ell,
\end{equation}
where the random function $\bex(t)$ is a Brownian excursion.

Moreover, the convergence \eqref{t1a} holds jointly for any set of $\gs$,
with $W_\gs$ given by \eqref{t1w} with the same $\bex$ for all $\gs$.

All moments of $W_\gs$ are finite, and
all moments (including mixed moments) converge in \eqref{t1a}.
In particular,
\begin{align}
  \E[\ns(\pinwww)]&\sim \E[W_\gs]n^{(m+\ell)/2}, \label{t1e}
\\
  \Var[\ns(\pinwww)]&\sim \Var[W_\gs]n^{m+\ell}. \label{t1var}
\end{align}
\end{theorem}

\begin{example}\label{E21}
  Let $\gs=\perm{21}$. Then $\wpp{21}=2\qqw$ by \refE{Ewgs}; hence
  \eqref{t1a}--\eqref{t1w}, with $\ell=1$ and $m_1=m=2$, 
yield a new proof of \eqref{qj}.
\end{example}

We note two special cases when the multiple integral in \eqref{t1w} reduces
to a single integral.

\begin{example}\label{E1}
If $\gs$ is indecomposable,
i.e.,  has only one block (see \refSS{SSblocks}), 
\eqref{t1w} yields
\begin{equation}\label{e1}
W_\gs=w_\gs\intoi \bex(t)^{m-1}\dd t.  
\end{equation}
The special case $m=2$ (i.e., $\gs=\perm{21}$) 
yields, as said in \refE{E21},   the Brownian excursion area in \eqref{qj},
which has been intensely studied, see \eg{} \cite{SJ201} and the references
there.

The case $m=3$ (i.e., $\perm{231}$ or $\perm{312}$)
yields the random variable
$\intoi\bex(t)^2\dd t$, 
which has been studied before  by \citet{Nguyen};
among other results he found a simple formula for the Laplace transform, 
which 
as noted in \cite[Example 7.17]{SJ290} shows that
the limit $W_\gs$ in this case,
ignoring the constant factor $w_\gs$, 
has the distribution denoted $S_{3/2}$ by 
\citet{BianePY}.

The integral in \eqref{e1} for a general $m$ has been studied by
\citet{Richard}.
\end{example}

\begin{example}\label{Elika}
  If all blocks have the same size $m_1=\dots=m_\ell$, then, by symmetry,
\eqref{t1w} yields
\begin{equation}\label{elika}
  W_\gs=\frac{w_\gs}{\ell!}\biggpar{\intoi\bex(t)^{m_1-1}\dd t}^\ell.
\end{equation}
Cf.\ \refE{E1} (the special case $\ell=1$), and see again \cite{SJ201,Nguyen,Richard}.
In particular, if all blocks have size 2, then $W_\gs$ is a constant times a
power of the Brownian excursion area.
\end{example}

\refT{T1} should be compared to the similar result for \perm{132}-avoiding
permutations in \cite[Theorem 2.1]{SJ290}, where also the limiting
distributions can be expressed using a Brownian excursion, but in general
in a much more complicated way, see \cite[Section 7]{SJ290}.
(At least, we do not know any simpler descriptions of those  limit
variables,
although it is conceivable that such might exist.)
In particular,  the limits in \eqref{e1} appear also as limits for
\zzz-avoiding permutations, see \cite[Examples 7.6--7.8]{SJ290}.

\begin{remark}\label{Rnormal}
The results obtained here for random \perm{321}-avoiding permutations,
and in \cite{SJ290} for
random \perm{132}-avoiding permutations, are very different from
the non-restricted case of uniformly random permutations in $\fS_n$:
it is well-known that if $\pix$ is a uniformly random permutation in $\fS_n$,
then $\ns(\pix)$ has an asymptotic normal distribution as \ntoo{}
for every fixed permutation $\gs$, and that (as a consequence)
$\ns(\pix)$ is concentrated around its mean in the sense that
$\ns(\pix)/\E[\ns(\pix)]\pto1$ as \ntoo.
See \citet{Bona-Normal,Bona3}
and \citet{SJ287}. 
\end{remark}

The moment convergence in \refT{T1}
yields the asymptotic formula \eqref{t1e} for the expectation,
involving the constant
\begin{equation}\label{et1}
\E W_\gs
= w_\gs\int_{0<t_1<\dots<t_\ell<1} 
\E\bigsqpar{\bex(t_1)^{m_1-1}\dotsm \bex(t_\ell)^{m_\ell-1}}\dd t_1\dotsm \dd t_\ell.
\end{equation}
We do not know any general formula for this integral, 
but it can be computed in many cases, 
and often higher moments too, see \refS{Smoments}.
In particular, we obtain the following: 
\begin{corollary}\label{C1}
  If\/ $\gs\in\fSmwww$ is an indecomposable \www-avoiding permutation, then,
  as \ntoo, 
with $w_\gs$ given by \eqref{wgs},
  \begin{equation}\label{c1}
    \E [\ns(\pinx\www)]\sim 
(\E W_\gs) n^{(m+\ell)/2} =
w_\gs 2^{-(m-1)/2} \gG\bigpar{\tfrac{m+1}2}
 n^{(m+1)/2}.
  \end{equation}
Similarly, \eqref{t1var} holds with
\begin{equation}\label{c1var}
\Var W_\gs
=
w_\gs^22^{1-m}\biggpar{\frac{2(m-1)!}{m}\Bigpar{1-\frac{m!^2}{(2m)!}}
-\gG\Bigpar{\frac{m+1}2}^2}.  
\end{equation}
\end{corollary}

\begin{corollary}
  \label{C2}
If $\gs$ has two blocks, of lengths $m_1$ and $m_2$, then
\eqref{t1e} holds
with $\ell=2$ and
\begin{equation}\label{c2w}
\E W_\gs
\\
=
w_\gs
2^{-m/2}\frac{m}{m_1m_2}\Bigpar{1-\frac{m_1!\,m_2!}{m!}}
\gG\Bigpar{\frac{m}2},
\end{equation}
where $m=|\gs|=m_1+m_2$ and $w_\gs$ is given by \eqref{wgs}.
\end{corollary}

In particular, in the cases $\gs=\perm{21}$, \perm{231},  \perm{312},
\refC{C1} yields, using the values of $w_\gs$ in \refE{Ewgs},
 the asymptotics in \eqref{e12H}--\eqref{e132H} obtained
from \cite{CEF} and \cite{Homberger}.
Similarly, \eqref{e231H} 
follows from \refC{C2} (or by the method in \refE{E1243}),
and 
\eqref{e321H} follows trivially from \eqref{t1e} since \eqref{t1w} yields 
$\Wpp{123}=1/6$.

\begin{remark}\label{Rflera}
  The general problem can be generalized to permutations avoiding a given set of
permutations.  Define, extending \eqref{fsgt},
$\fS_n(\tau_1,\dots,\tau_k):=\bigcap_{i=1}^k\fS_n(\tau_i)$,
and let $\pinxx{\tau_1,\dots,\tau_k}$ be a uniformly random permutation in
the set $\fS_n(\tau_1,\dots,\tau_k)$.
The size 
$|\fS_n(\tau_1,\dots,\tau_k)|$ was found for all cases 
with $k\ge2$ and all $|\tau_i|=3$
by \citet{Simion-Schmidt};
we give
some simple results on the asymptotic distribution of
$\ns(\pinxx{\tau_1,\dots,\tau_k})$ 
for these cases in \cite{SJ-flera}.
Somewhat surprisingly, there are cases with an asymptotic normal
distribution similar to the one for random unrestricted permutations
(see \refR{Rnormal}), and thus quite different from the limiting
distributions for $\ns(\pint)$ for a single $\tau$ with $|\tau|=3$ in the
present paper and \cite{SJ290}.

\end{remark}

In the present paper we study only the numbers $\ns$ of occurences of some
pattern in $\pint$.
There is also a number of papers studing other properties of 
random $\tau$-avoiding permutations.
Some examples, in addition to those mentioned above, are
consecutive patterns \cite{BarnabeiCS-consecutive};
descents and the major index \cite{BarnabeiBES-descent}; 
number  of fixed points 
\cite{RobertsonSZ,Elizalde-PhD,Elizalde-fixed, 
MinerPak,HRS-fixed};  
position of fixed points 
\cite{MinerPak,HRS-fixed}; 
exceedances
\cite{Elizalde-PhD,Elizalde-fixed}; 
longest increasing subsequence \cite{DeutschHW}; 
shape 
and distribution of individual values $\pi_i$
\cite{MadrasLiu, 
MadrasPehlivan,  
HRS-I}. 

\section{Preliminaries}

\subsection{Compositions and decompositions of permutations}
\label{SSblocks}
If $\gs\in\fS_m$ and $\tau\in\fS_n$,  their \emph{composition}
$\gs*\tau\in\fS_{m+n}$ is defined
by letting $\tau$ act on $[m+1,m+n]$ in the natural
way; more formally,
$\gs*\tau=\pi\in\fS_{m+n}$ where $\pi_i=\gs_i$ for $1\le i\le m$, and
$\pi_{j+m}=\tau_j+m$ for $1\le j\le n$. 
It is easily seen that $*$ is an associative operation that makes $\fS_*$
into a semigroup (without unit, since we only consider permutations of
length $\ge1$). We say that a permutation $\pi\in\fS_*$  is
\emph{decomposable} if $\pi=\gs*\tau$ for some $\gs,\tau\in\fS_*$, and
\emph{indecomposable} otherwise;
we also call an indecomposable permutation a
\emph{block}.
Equivalently, $\pi\in\fS_n$ is decomposable if and only if $\pi:[m]\to[m]$
for some $1\le m<n$. 
See \eg{}  \cite[Exercise VI.14]{Comtet}.

It is easy to see that any permutation $\pi\in\fS_*$ has a unique
decomposition $\pi=\pi_1*\dots*\pi_\ell$  into indecomposable permutations
(blocks)  $\pi_1,\dots,\pi_\ell$ (for some, unique, $\ell\ge1$); 
we call these the \emph{blocks of $\pi$}.

An \emph{inversion} in a permutation $\pi$ is an occurence $(i,j)$ of the
pattern $\perm{21}$.
Given a permutation $\pi\in\fS_n$, its \emph{inversion graph} $\gG_\pi$ is
the graph with vertex set $[n]$, and an edge $ij$ for every inversion
$(i,j)$ in $\pi$.
(This is the same as the intersection graph of the set of line segments
$[(i,0),(\pi_i,1)]\subset\bbR^2$.
The graphs isomorphic to $\gG_\pi$ for some permutation $\pi$ are known as
\emph{permutation graphs}, see \eg{} \cite{Brand}.)

It is easy to see that the connected components of the inversion graph
$\gG_\pi$ are precisely the blocks of $\pi$; in particular, $\pi$ is
indecomposable if and only if $\gG_\pi$ is connected, see \cite{KohRee}.

\subsection{\www-avoiding permutations}

Given any permutation $\pi\in\fS_n$, let 
\begin{align}
\Ep&=\Ep(\pi):=\set{i\in[n]:\pi_i>i},
\\
\Em&=\Em(\pi):=[n]\setminus\Ep(\pi)=\set{i\in[n]:\pi_i\le i}.  
\end{align}
Thus $\Ep$ and $\Em$ form a partition of $[n]$.
$\Ep$ is known as the set   of \emph{exceedances} of $\pi$.

It is well-known that a permutation $\pi$ is \www-avoiding
if and only if $\pi$ is the union of two increasing subsequences,
and in
particular, if $\pi\in\fSxwww$, 
then the subsequences with indices in $\Ep$ and $\Em$ are increasing.
(This is easy to see directly; it also follows from the BJS bijection
in \refSS{SSBJS}.)
In other words, if $i<j$ and $i,j\in \Ep$ or $i,j\in \Em$, then $\pi_i<\pi_j$.
Furthermore, if $i<j$ and $i\in\Em$, $j\in\Ep$, then $\pi_i\le i<j<\pi_j$.
Consequently:
\begin{equation}\label{inv+-}
 \text{If $(i,j)$ is an inversion in $\pi\in\fSxwww$, 
then $i\in\Ep(\pi)$ and $j\in\Em(\pi)$.}
\end{equation}

\subsection{Dyck paths and the BJS bijection} \label{SSBJS}

A \emph{Dyck path} of length $2n\ge0$ is a mapping 
$\gam:\set{0,\dots,2n}\to\bbZ$ such that $\gam(0)=\gam(2n)=0$, $\gam(x)\ge0$
for every $x$, and $|\gam(x+1)-\gam(x)|=1$ for all $x\in\set{0,\dots,2n-1}$.
We identify a Dyck path with the corresponding continuous
function $\gam:[0,2n]\to\bbR$
obtained by linear interpolation.
Let $\Dyckn$ be the set of Dyck paths of length $2n$.
It is well-known that $|\Dyckn|=C_n$, the $n$-th
Catalan number in \eqref{catalan}, see \eg{}
\cite[Exercise 6.19(i)]{StanleyII}.

We use, as \cite{HRS-I,HRS-II}, a bijection between 
$\Dyckn$ and $\fSnwww$, i.e., between
Dyck paths of length $2n$
and $\www$-avoiding permutations of length $n$;
the bijection is known as the BJS bijection after \citet{BJS}
and can be described as follows. (See also \cite{Callan} for more on this
and on other bijections between $\Dyckn$ and $\fSnwww$.)

Fix a Dyck path $\gamma\in\Dyck_{2n}$, and let $m$ be the number of
increases (or decreases) in $\gam$.
Let $a_i\ge1$ be the length of the $i$-th run of increases, 
and let $d_i\ge1$ be the length of the $i$-th run of decreases in $\gam$.
Let, for $0\le i\le m$, $A_i:=\sum_{j=1}^i a_j$ and $D_i:=\sum_{j=1}^i d_j$;
let $\cA:=\set{A_i:1\le i\le m-1}$,
$\cA_1:=\set{A_i+1:1\le i\le m-1}$,    
 $\cD:=\set{D_i:1\le i\le m-1}$,
$\cA_1\cp:=[n]\setminus\cA_1$, and
$\cD\cp:=[n]\setminus\cD$. 
Finally, define the permutation $\pi_\gam\in\fS_n$ as the unique permutation
with $\pi:\cD\to\cA_1$, and therefore $\pi:\cD\cp\to\cA_1\cp$, such that
$\pi$ is increasing on $\cD$ and on $\cD\cp$.
(In particular, $\pi_\gam(D_i)=A_i+1$ for $1\le i\le m-1$.)

Then, $\gam\to\pi_\gam$ is a bijection of $\Dyck_n$ onto $\fSnwww$, see e.g.\
\cite{BJS,Callan}.
Moreover \cite[Lemma 2.1]{HRS-I},
\begin{equation}\label{exc}
  \Ep(\pi_\gam)=\cD(\gam),
\qquad
  \Em(\pi_\gam)=\cD\cp(\gam).
\end{equation}

We define also, as in \cite{HRS-I},
\begin{equation}\label{yi}
  y_i:=A_i-D_i=\gam(A_i+D_i).
\end{equation}

\subsection{Brownian excursion}

A (normalized) Brownian excursion $\bex(t)$ 
is a random continuous function on $\oi$ that
can  be defined as a Brownian motion $B(t)$ conditioned on
$B(1)=B(0)=0$ and 
$B(t)\ge0$, $t\in\oi$; since this means conditioning on an event of
probability zero, the conditioning has to be interpreted with some care,
e.g.\ as a suitable limit.
See also \cite[Chapter XII]{RY} for an alternative definition.

The distribution of the Brownian excursion $\bex$ has also several other
descriptions; for example,
$\bex$ has
the same distribution as a
Bessel bridge of dimension 3 over $\oi$,
see \eg{} \cite[Theorem XII.(4.2)]{RY}
and thus also as the absolute value of a 3-dimensional Brownian bridge,
i.e.,
\begin{equation}
  \label{exbr}
\bex(t)\eqd\sqrt{\brbr_1(t)^2+\brbr_2(t)^2+\brbr_3(t)^2},
\qquad t\in\oi,
\end{equation}
where $\brbr_1,\dots\brbr_3$
are independent Brownian bridges.

\subsection{Some notation}

$\gl_d$ denotes $d$-dimensional Lebesgue measure.

For typographical reasons, we sometimes write $\pi(i)$ for $\pi_i$.

We say that an event $\cE_n$ (depending on $n$, \eg{} through $\pinwww$) 
holds \emph{\whp} if $\P(\cE_n)\to1$ as
\ntoo, and \emph{\vhp} if $\P(\cE_n)=1-O(e^{-n^c})$ for some $c>0$;
note that the latter implies $\P(\cE_n)=1-O(n^{-C})$ for any $C>0$.

We let $c$ and $C$, possibly with subscripts, denote unspecified
positive constants that may depend on $\gs$; they may
vary between different occurrences.

\section{The parameter $w_\gs$}\label{Swgs}

Let $\gs$ be a \perm{321}-avoiding permutation.

First, assume that $\gs$ is a  block with $m=|\gs|>1$.
In this case, let $\Pi_\gs$ be the set of all vectors
$(x_2,\dots,x_m)\in\ooo^{m-1}$ such that, with $x_1=0$,
\begin{romenumerate}
\item \label{pi1}
$0=x_1\le x_2\le\dots\le x_m$;
\item \label{pi2}
If $i<j$, $i\in \Ep(\gs)$ and $j\in\Em(\gs)$, then
\begin{enumerate}
 \renewcommand{\labelenumii}{\textup{(\alph{enumii})}}%
 \renewcommand{\theenumii}{\textup{(\alph{enumii})}}%
\item \label{pia}
if $\gs_i<\gs_j$, then $x_j\ge x_i+1$;
\item \label{pib}
if $\gs_i>\gs_j$, then $x_j\le x_i+1$.
\end{enumerate}
\end{romenumerate}
By \eqref{inv+-},
\ref{pib} applies whenever $(i,j)$ is an inversion in $\gs$.
Hence, $|x_i-x_j|\le1$ whenever 
$ij$ is an edge in the inversion graph $\gG_\gs$, and since the inversion
graph is connected (because $\gs$ is assumed to be a block), it follows that
\begin{equation}
  |x_i|=|x_i-x_1|\le m-1
\end{equation}
for every $i\le m$.
Consequently, the set $\Pi_\gs$ is bounded, and since it is defined as an
intersection of closed half-planes, $\Pi_\gs$ is compact and a polytope.
It is also easy to see that $\Pi_\gs$ has a nonempty interior 
$\Pigso$,
obtained by
taking strict inequalities in \ref{pi1}--\ref{pi2}.
Let $v_\gs:=\gl_{m-1}(\Pi_\gs)$,  the volume of the polytope $\Pi_\gs$; thus
$0<v_\gs<\infty$. 


Next, for  a \www-avoiding block $\gs$,  let
\begin{equation}
  \label{wgsblock}
w_\gs:=
\begin{cases}
2^{(|\gs|-3)/2}v_\gs, & \text{$\gs$ is a block with $|\gs|>1$},
\\
1, & |\gs|=1.
\end{cases}
\end{equation}
Finally, for an arbitrary \www-avoiding permutation $\gs$ with blocks
$\gs^1,\dots,\gs^\ell$, define
\begin{equation}\label{wgs}
  w_{\gs}:=\prod_{i=1}^\ell w_{\gs^i}.
\end{equation}

\begin{example}\label{Evgs}
For $\gs=\perm{21}$, we only have to consider the case $i=1$, $j=2$ for
\ref{pi2} in the definition of $\Pi_\gs$; 
in this case \ref{pib} applies, and yields $x_2\le1$.
Together with \ref{pi1} we obtain $0\le x_2\le 1$, so $\Pi_{21}=[0,1]$,
and 
\begin{equation}
\vpp{21}=1.  
\end{equation}

For both $\gs=\perm{231}$ and $\gs=\perm{312}$, we similarly obtain 
$\Pi_\gs:\set{(x_2,x_3):0\le  x_2\le x_3\le 1}$. Thus
  \begin{equation}
    \vpp{231}=\vpp{312}=\tfrac12.
  \end{equation}

Similarly, elementary calculations show that for the 5 blocks in
$\fS_4(\www)$,
\begin{equation}
  \vpp{2341}=  
  \vpp{2413}=  
  \vpp{3142}=  
  \vpp{3412}=  
  \vpp{4123}=  
\tfrac{1}6.
\end{equation}
However, for longer blocks, $v_\gs$ depends not only on the length $|\gs|$.
For example, omitting the calculations,
\begin{align}
&\vpp{23451}=  \vpp{51234} = \tfrac{1}{24},
&&
  \vpp{24153} = \tfrac{2}{24},
\\
&\vpp{234561}=  \vpp{612345} = \tfrac{1}{120},
&&
  \vpp{315264} = \tfrac{5}{120}.
\end{align}
\end{example}

\begin{problem}
  Based on these and other similar examples, we conjecture that for every
  block $\gs\in\fSxwww$, $v_\gs=\nu_\gs/(|\gs|-1)!$ for some integer
  $\nu_\gs\ge1$. 
Prove this!
Moreover, if this holds,  find a combinatorial
interpretation of $\nu_\gs$.
\end{problem}

\begin{example}
\label{Ewgs}
The values for $v_\gs$ in \refE{Evgs} 
yield by
\eqref{wgsblock}--\eqref{wgs}
\begin{align}
\wpp{21}&=2\qqw \vpp{21}=1/\sqrt2, \label{w21} 
\\
    \wpp{231}&=\wpp{312}=\xfrac12, \label{w231}
\\
\wpp{132}&=\wpp{213}=\wpp{1}\wpp{21}=1/\sqrt2, 
\\
\wpp{123}&=\wpp{1}\wpp{1}\wpp{1}=1.\label{w123}
\end{align}
As said above, \eqref{w21}--\eqref{w123} combine with \refC{C1} and
\eqref{et1} to yield  
\eqref{e12H}--\eqref{e321H};
furthermore, \eqref{w21} and \refT{T1} yield  \eqref{qj}.
\end{example}

\section{Proof of \refT{T1}}\label{Spf}

The proof of \refT{T1} is rather long, and will be interspersed with several
lemmas. 

Suppose that $\gs\in\fSxwww$ is fixed and that $\pi\in\fSnwww$
(for a large $n$)
Consider first the case when $\gs$ is a block. 

\begin{lemma}\label{L1}
Suppose that $\gs\in\fSmwww$ is a block with $m=|\gs|>1$.
If $\pi\in\fSnwww$ and $1\le k_1<\dots<k_m\le n$,
then $\bk:=(k_1,\dots,k_m)$ is an occurrence of $\gs$ in $\pi$
if and only if:
\begin{romenumerate}
\item\label{L1+}
 $k_i\in\Ep(\pi)$ for $i\in\Ep(\gs)$;
\item \label{L1-}
$k_i\in\Em(\pi)$ for $i\in\Em(\gs)$;
\item \label{L1inv}
if\/ $i<j$ with $i\in\Ep(\gs)$ and $j\in\Em(\gs)$, then
\begin{equation}\label{l1}
  \pi_{k_i}>\pi_{k_j} \iff \gs_i >\gs_j.
\end{equation}
\end{romenumerate}  
\end{lemma}

\begin{proof}
Note first that by definition, $\bk$ is an occurrence of $\gs$ if and only
if \eqref{l1} holds for every pair $(i,j)$ with $1\le i<j\le m$; the point of 
\ref{L1inv} is that we only have to check this for certain pairs $(i,j)$.

$\Longrightarrow$:
Suppose  that $\bk$ is an occurrence of $\gs$.

Let $i\in\Ep(\gs)$. Since $\gs$ is a block, its inversion graph $\gG_\gs$ is
connected. Hence there is an inversion $(i,j)$ for some $j>i$ or an
inversion $(j,i)$ for some $j<i$, but the latter is impossible when
$i\in\Ep$ by \eqref{inv+-}. Consequently there is an inversion $(i,j)$ in $\gs$,
and then $(k_i,k_j)$ must be an inversion in $\pi$; in particular,
$k_i\in\Ep(\pi)$ by \eqref{inv+-}. Hence \ref{L1+} holds.

The proof of \ref{L1-} is similar.

Finally, \eqref{l1} holds, as noted above, for all pairs $(i,j)$ with $i<j$.

$\Longleftarrow$:
Conversely, suppose that \ref{L1+}--\ref{L1inv} hold,
and let $i<j$.
If $i,j\in\Ep(\gs)$, 
then $k_i,k_j\in\Ep(\pi)$ by \ref{L1+}, and thus 
\eqref{inv+-} implies that both $\gs_i<\gs_j$ and
$\pi_{k_i}<\pi_{k_j}$; hence \eqref{l1} holds in this case.
Similarly, \eqref{l1} holds if $i,j\in\Em(\gs)$, or if
$i\in \Em(\gs)$ and $j\in\Ep(\gs)$.
Finally, in the remaining case
$i\in \Ep(\gs)$ and $j\in\Em(\gs)$, \ref{L1inv} applies.
Hence, \eqref{l1} holds for every pair $(i,j)$ with $1\le i<j\le m$,
and thus $\bk$ is an occurrence of $\gs$.
\end{proof}

Let 
\begin{equation}\label{gDi}
  \gD_i=\gD_i(\pi):=\pi_i-i,
\qquad i\in[n],
\end{equation}
and note that $\gD_i>0$ if $i\in \Ep(\pi)$ and
$\gD_i\le 0$ if $i\in\Em(\pi)$.

\begin{lemma}\label{L1a}
\refL{L1} holds also if \eqref{l1} in \ref{L1inv} is replaced by
\begin{equation}\label{mask}
\gs_i > \gs_j
\iff
  k_j-k_i < |\gD_{k_i}|+|\gD_{k_j}|.
\end{equation}
\end{lemma}

\begin{proof}
Suppose that \ref{L1+}--\ref{L1-} hold, and that $i<j$ with $i\in\Ep(\gs)$
and $j\in\Em(\gs)$. 
Then   
 $k_i\in\Ep(\pi)$ and $k_j\in\Em(\pi)$, and thus
\begin{equation}
  \pi_{k_i}-\pi_{k_j}=k_i-k_j+\gD_{k_i}-\gD_{k_j}
=k_i-k_j+|\gD_{k_i}|+|\gD_{k_j}|.
\end{equation}
Consequently, \eqref{l1} holds if and only if \eqref{mask} does.
\end{proof}

Before proceeding, we use \refL{L1a} to give a useful upper bound for
$\ns(\pi)$. 
Let
\begin{equation}\label{gDX}
\gDX=  \gDX(\pi):=\max_{1\le i\le n}|\gD_i|.
\end{equation}
Then, \eqref{mask} implies that
$0\le k_j-k_i\le 2\gDX$ when $(i,j)$ is an inversion in $\gs$.
Since the inversion graph $\gG_\gs$ is connected, this implies 
\begin{equation}\label{tk}
  0< k_i-k_1\le 2m\gDX, 
\qquad i=2,\dots,m.
\end{equation}
Hence, the number of occurrences $\bk$ of $\gs$ with a given choice of $k_1$
is at most $(2m\gDX)^{m-1}$, and thus 
\begin{equation}\label{nsbound}
  \ns(\pi)\le (2m)^{m-1}n\gDX^{m-1}
 =O(n\gDX^{m-1}).
\end{equation}

Now let $\pi=\pinwww$ be random.
By the BJS bijection, the uniformly random $\pinwww$ corresponds to a
uniformly random Dyck path $\gam\in\Dyckn$ by $\pinwww=\pi_\gam$.
We use the notation in \refSS{SSBJS}; we sometimes write $\gam$, $\pi=\pig$,
or $\gs$ as arguments of various sets or quantities for clarity, but often
we omit them. 

It is well-known that a random Dyck path converges in distribution to a
Brownian excursion after suitable normalization as \ntoo. To be precise,
\begin{equation}
  \label{dex}
\frac{\gam(2nt)}{\sqrt{2n}} \dto \bex(t)
\end{equation}
as random elements of $C\oi$, see \cite{Kaigh}.
We use
the Skorohod coupling theorem \cite[Theorem 4.30]{Kallenberg}, 
and may thus assume in the remainder of the proof
that the Dyck paths, and thus the permutations
$\pinwww$, are coupled for different $n$ such that
\eqref{dex} holds a.s.
In other words,
\begin{equation}\label{dex2}
  \gam(i)=\sqrt{2n} \Bigpar{\bex\Bigpar{\frac{i}{2n}}+o(1)},
\end{equation}
where, as throughout this proof, 
$o(1)\to0$ as \ntoo, uniformly in $i\in[n]$ (and in other similar variables
later).
However, the $o(1)$ may depend on the random $\pinwww$, $\gamma$ and
$\bex$.   $O(\dots)$ below is interpreted similarly.

\citet[Section 2]{HRS-I} show that a random Dyck path \vhp{}
satisfies some regularity properties there called \PC,
moreover, they show some deterministic consequences of these properties 
(at least for large $n$). 
%
By the Borel--Cantelli lemma, the  \PC{} thus \as{} hold for all large
$n$,
so we may assume that these conditions and their consequences hold for $\gamma$.

In particular, by \cite[Lemma 2.7]{HRS-I},
if $j\in\cD$, then
$|\pi_\gam(j)-j-\gam(2j)|<10 n^{0.4}$,
while if
$j\notin\cD$, then
$|\pi_\gam(j)-j+\gam(2j)|<10 n^{0.4}$. 
Hence, 
recalling the notation \eqref{gDi} (with $\pi=\pi_\gam$) and
using \eqref{dex2} and \eqref{exc},
\begin{equation}\label{dja}
  \gD_j = 
  \begin{cases}
\gam(2j) + O(n^{0.4})=\sqrt{2n}\, \bex(j/n) +o\bigpar{n\qq},& j\in\cD=\Ep(\pig).
\\
-\gam(2j) + O(n^{0.4})=-\sqrt{2n}\, \bex(j/n) +o\bigpar{n\qq},& j\in\cD\cp=\Em(\pig),
  \end{cases}
\end{equation}
and consequently, for all $j\in[n]$,
\begin{equation}\label{djb}
 | \gD_j| 
=\sqrt{2n}\, \bex(j/n) +o\bigpar{n\qq}.
\end{equation}
Note  that \eqref{djb} implies, by the definition \eqref{gDX},
\begin{equation}\label{djc}
 \gDX 
= O\bigpar{n\qq}.
\end{equation}

Let, for $k\in[n]$, $\cA_k$ be the set of all occurrences $\bk=\bkk$ of
$\gs$ in $\pi_\gam$ such that $k_1=k$. Thus $\ns(\pi_\gam) =\sumkn|\cA_k|$.

We have shown above that if $\bk\in\cA_k$, then 
\eqref{tk} holds, and thus, using \eqref{djc}, 
\begin{equation}\label{skrapan}
|k_i-k|=O(\gDX)=O\bigpar{n\qq}=o(n).  
\end{equation}
Since $\bex(t)$ is continuous, it thus follows from \eqref{djb} that
\begin{equation}\label{djd}
 | \gD_{k_i}| 
=\sqrt{2n}\, \bex(k_i/n) +o\bigpar{n\qq}
=\sqrt{2n}\, \bex(k/n) +o\bigpar{n\qq}.
\end{equation}
Hence, in \eqref{mask}, we have
\begin{equation}\label{maskar}
|\gD_{k_i}|+|\gD_{k_j}|=2\qqc n\qq\bex(k/n)+o\nqq.  
\end{equation}
Motivated by \eqref{maskar}, 
let $\cA'_k$ be the set of $m$-tuples $\bk=\bkk$ with $k=k_1<\dots<k_m$
such that \refL{L1}\ref{L1+}--\ref{L1-} hold, and, furthermore,
for every $i\in\Ep(\gs)$ and $j\in\Em(\gs)$ with $i<j$, 
\begin{equation}\label{orm}
\gs_i > \gs_j
\iff
  k_j-k_i <
2\qqc n\qq\bex(k/n).
\end{equation}
Note that this agrees with the characterization of $\cA_k$ implied by \refL{L1a}
except that the bound $|\gD_{k_i}|+|\gD_{k_j}|$ in \eqref{mask} is replaced
by $2\qqc n\qq\bex(k/n)$.
Consequently, if $\bk\in\cA_k\gD\cA'_k$, then for some pair $(i,j)$ either
\begin{equation}
  \label{bad}
|\gD_{k_i}|+|\gD_{k_j}| \le k_j-k_i\le 2\qqc n\qq\bex(k/n)
\end{equation}
 or conversely.

Furthermore, if $\bk\in\cA'_k$, then \eqref{orm} shows that 
\begin{equation}
0\le k_j-k_i\le 2\qqc n\qq\max_t\bex(t)=O\nqq  
\end{equation}
for every inversion $(i,j)$ of $\gs$, and
thus, by the argument used above for \eqref{tk}, $k_i-k=O\nqq$ for every
$i\le m$; as a consequence, \eqref{djd} holds for $\bk\in\cA'_k$ too.

It follows that if $\bk\in\cA_k\gD\cA'_k$, then 
$|k_i-k|=O\nqq$ for $i=2,\dots,m$, and furthermore,
for some pair $(i,j)$ with $1\le i<j\le m$,
\begin{equation}
\bigabs{k_j-k_i- 2\qqc n\qq\bex(k/n)}=o\nqq.
\end{equation}
It follows that
\begin{equation}\label{fela}
  \bigabs{\cA_k\gD\cA'_k} = o\bigpar{n^{(m-1)/2}}.
\end{equation}
Hence, we may in the sequel consider $\cA'_k$ instead of $\cA_k$.

Next, 
let $\cao$ 
be the set of $m$-tuples $\bk=\bkk\in[1,n]^m$
such that \refL{L1}\ref{L1+}--\ref{L1-} hold (with $\pi=\pig$);
i.e.,
\begin{equation}
  \label{cao}
\cao=\prod_{i=1}^m E_{\eps_i}(\pig), 
\end{equation}
where $\eps_i\in\set{+,-}$ is such that 
$i\in E_{\eps_i}(\gs)$.
Note also that $\gs_1>1$ since $\gs$ is a block of length $>1$, and thus
$1\in E_+(\gs)$, i.e., $\eps_1=+$.

Furthermore,
let $\cB_k$
be the set of $m$-tuples $\bk=\bkk\in[1,n]^m$ 
such that $k_1=k$ and
\begin{equation}\label{bk}
  (k_2-k,\dots,k_m-k)\in 
\cP_k:=
2\qqc n\qq\bex(k/n) \Pi_\gs,
\end{equation}
where $\Pi_\gs$ is the polytope defined in \refS{Swgs};
note that this means that $k=k_1\le k_2\dots\le k_m$ and that
the equivalences \eqref{orm} hold (for pairs $(i,j)$ as above), 
except  in some cases of equality.
Consequently, $\cA'_k$ equals $\cao\cap\cB_k$, except possibly for some
points on the boundary, and thus, recalling \eqref{fela}.
\begin{equation}\label{akka}
  |\cA_k|
=|\cA'_k|+o\bigpar{n^{(m-1)/2}}
=|\cao\cap\cB_k|+o\bigpar{n^{(m-1)/2}}.
\end{equation}
Furthermore, \eqref{bk} implies that,
recalling $\gl_{m-1}(\Pi_\gs)=v_\gs$,
\begin{equation}\label{bk2}
  \begin{split}
  |\cB_k| &=
\gl_{m-1} \bigpar{\cP_k} + O\bigpar{n^{(m-2)/2}}
\\&
=  \bigpar{2\qqc n\qq\bex(k/n)}^{m-1} v_\gs + O\bigpar{n^{(m-2)/2}}.
  \end{split}
\end{equation}
The idea is now that roughly
each second point belongs to $\Ep(\pig)$ and
each second to $\Em(\pig)$, and thus $|\cao\cap\cB_k|\approx 2^{-(m-1)}|\cB_k|$.
We make this precise in the following lemma.
\begin{lemma}\label{Lab}
If the \PC{} hold for $\gam$, and $1\le a\le b \le n$, then
\begin{equation}\label{lab}
|[a,b]\cap\cD| =\tfrac12(b-a) + O\bigpar{(b-a)^{0.6}+n^{0.18}}.  
\end{equation}
\end{lemma}
Here, the $O(\dots)$ is uniform in all such $\gam$, $a$ and $b$.

\begin{proof}
Since \eqref{lab} is trivial for small $n$, we may assume that $n$ is large
enough when needed below.

The \PC{} \cite[Definition 2.3]{HRS-I} include that if $|j-i|\ge n^{0.3}$, then 
\begin{equation}
  \label{PCdij}
|D_j-D_i-2(j-i)|<0.1|i-j|^{0.6}.
\end{equation}
If $|j-i|< n^{0.3}$, let $\ell=\min(i,j)-\ceil{n^{0.3}}$ or
$\ell=\max(i,j)+\ceil{n^{0.3}}$, chosen such that $\ell\in[1,n]$.
Then, by \eqref{PCdij} for the pairs $(i,\ell)$ and $(j,\ell)$ and the
triangle inequality,
\begin{equation}  \label{PCdij2}
|D_j-D_i-2(j-i)|<0.2\bigpar{|j-i|+\ceil{n^{0.3}}}^{0.6}
=O\bigpar{|j-i|^{0.6}+n^{0.18}}.
\end{equation}

Now,
let $i$ and $j$ be such that 
$D_{i-1}<a \le D_i$ and $D_{j-1}<b\le D_{j}$.
Then $[a,b)\cap\cD=\set{D_i,\dots,D_{j-1}}$ and thus 
$|[a,b)\cap\cD|=j-i$.
Furthermore, by \cite[Lemma 2.5]{HRS-I},
\begin{equation}\label{adb}
  |a-D_i|\le|D_i-D_{i-1}|\le n^{0.18},
\qquad
  |b-D_j|\le|D_j-D_{j-1}|\le n^{0.18}.
\end{equation}
Consequently, using \eqref{adb} together with \eqref{PCdij} or \eqref{PCdij2},
\begin{equation}
  b-a = D_j-D_i + O\bigpar{n^{0.18}}
= 2(j-i) + O\bigpar{|j-i|^{0.6}+n^{0.18}}.
\end{equation}
This yields \eqref{lab}, since (for the error term)
either $j-i=0$ or $j-i\le 1+D_{j-1}-D_i\le b-a$.
\end{proof}

Let $N:=\floor{n^{0.6}}$ and $i_\nu:=\floor{in/N}$, $0\le\nu\le N$.
Partition $(0,n]$ into $N$ intervals $I_\nu=(i_{\nu-1},i_\nu]$,
$1\le\nu\le N$, of lengths $|I_\nu|=n^{0.4}+O(1)$.

For $\nu_2 \dots\nu_m \in[N]$,
let
$\cQ_{k;\nu_2,\dots,\nu_m}:=\set{k}\times\prod_{j=2}^m I_{\nu_j}$,
and let
\begin{align}
\cN_k&:=    
\set{(\nu_2,\dots,\nu_m):\nu_2<\dots<\nu_m
\text{ and }\cQ_{k;\nu_2,\dots,\nu_m}\subseteq \cB_k}
\\
\cB_k''&:=\bigcup_{(\nu_2,\dots,\nu_m)\in\cN_k} \cQ_{k;\nu_2,\dots,\nu_m}.
\end{align}
Thus, $\cB''_k\subseteq\cB_k$. Furthermore, if $\bk=\bkk\in
\cB_k\setminus\cB''_k$, then $\bk\in Q_{k:\nu_2,\dots,\nu_m}$ for some 
$\nu_2,\dots,\nu_m$ such that either $\nu_j=\nu_{j+1}$ for some $j$, or
$Q_{k:\nu_2,\dots,\nu_m}\not\subseteq \cB_k$; in both cases,
the point in $\cP_k$ corresponding to $\bk$ by \eqref{bk} 
has distance $O\bigpar{n^{0.4}}$ to the boundary of $\cP_k$, and it follows
that
\begin{equation}\label{bk3}
  \begin{split}
  |\cB_k\setminus\cB''_k| &= O\bigpar{n^{(m-2)/2+0.4}}
= o\bigpar{n^{(m-1)/2}}.
  \end{split}
\end{equation}
Let
\begin{equation}\label{ca''}
  \cA''_k:=\cao\cap\cB''_k
= \bigcup_{(\nu_2,\dots,\nu_m)\in\cN_k} \cao\cap\cQ_{k;\nu_2,\dots,\nu_m}.
\end{equation}
Then $\cA''_k\subseteq\cao\cap\cB_k=\cA'_k$, and 
\begin{equation}\label{kn}
|\cA'_k\setminus\cA''_k|\le|\cB_k\setminus\cB''_k|= o\bigpar{n^{(m-1)/2}}.
\end{equation}
Furthermore, for each $(i_2,\dots,i_m)\in\cN_k$,
 \eqref{cao} shows that if 
$k\in E_+(\pig)=E_{\eps_1}(\pig)$, then
\begin{equation}\label{caoq}
  \cao\cap\cQ_{k;\nu_2,\dots,\nu_m}
=\set{k}\times\prod_{j=2}^m \bigpar{I_{\nu_j}\cap E_{\eps_j}(\pig)}.
\end{equation}
Furthermore, $E_+(\pig)=\cD$ and $E_-(\pig)=[n]\setminus\cD$ by \eqref{exc},
and thus \refL{Lab} shows that
$|I_{\nu_j}\cap E_{\eps_j}|=\frac12|I_{\nu_j}|\bigpar{1+o(1)}$ for every
$j$, regardless of the value of $\eps_j$.
Hence we obtain, using \eqref{ca''}, \eqref{caoq}, 
and the fact that $|\cB''_k|\le|\cB_k|=O(n^{(m-1)/2})$ by \eqref{bk2},
provided $k\in\Ep(\pig)$,
\begin{equation}\label{cam}
  \begin{split}
  |\cA''_k|
&= \sum_{(\nu_2,\dots,\nu_m)\in\cN_k}| \cao\cap\cQ_{k;\nu_2,\dots,\nu_m}|
\\&
= \sum_{(\nu_2,\dots,\nu_m)\in\cN_k}
\bigpar{2^{-(m-1)}+o(1)}|\cQ_{k;\nu_2,\dots,\nu_m}|
\\&
=  \bigpar{2^{-(m-1)}+o(1)}|\cB''_k|
=  2^{-(m-1)}|\cB''_k| + o\bigpar{n^{(m-1)/2}}.    
  \end{split}
\end{equation}
Consequently, by \eqref{akka}, \eqref{kn}, \eqref{cam}, \eqref{bk3},
\eqref{bk2}, 
\begin{equation}
  \begin{split}
  |\cA_k| 
&= |\cA'_k|+ o\bigpar{n^{(m-1)/2}}
= |\cA''_k|+ o\bigpar{n^{(m-1)/2}}
\\&
= 2^{-(m-1)}|\cB''_k|+ o\bigpar{n^{(m-1)/2}}
= 2^{-(m-1)}|\cB_k|+ o\bigpar{n^{(m-1)/2}}
\\&
=  2^{(m-1)/2}{n^{(m-1)/2}} \bex(k/n)^{m-1} v_\gs + o\bigpar{n^{(m-1)/2}},    
  \end{split}
\end{equation}
provided $k\in\E_+(\pig)$; otherwise $\cA_k=\emptyset$.

Finally, 
\begin{equation}\label{grig}
  \begin{split}
      \ns(\pig)
&= \sumkn|\cA_k|
\\&
=   2^{(m-1)/2} v_\gs {n^{(m-1)/2}}\sum_{k\in\Ep(\pig)} \bex(k/n)^{m-1} 
+ o\bigpar{n^{(m+1)/2}}.
  \end{split}
\end{equation}
For each interval $I_\nu$, using the continuity of $\bex$ and \refL{Lab},
\begin{equation}\label{gynt}
  \begin{split}
  \sum_{k\in\Ep(\pig)\cap I_\nu} \bex(k/n)^{m-1} 
&=   \sum_{k\in\Ep(\pig)\cap I_\nu} \bigpar{\bex(i_\nu/n)^{m-1} +o(1)} 
\\&
=   |\Ep(\pig)\cap I_\nu| \bigpar{\bex(i_\nu/n)^{m-1} +o(1)} 
\\&
=\tfrac12   |I_\nu| \bex(i_\nu/n)^{m-1} +o\bigpar{|I_\nu|}
\\&
=\tfrac12  \int_{I_\nu} \bex(x/n)^{m-1}\dd x +o\bigpar{|I_\nu|}.    
  \end{split}
\end{equation}
Summing over all $I_\nu$, we thus obtain by \eqref{grig},
recalling (and justifying) \eqref{wgsblock},
\begin{equation}\label{grieg}
  \begin{split}
      \ns(\pig)
&=   2^{(m-3)/2} v_\gs {n^{(m-1)/2}}\int_0^n \bex(x/n)^{m-1} \dd x
+ o\bigpar{n^{(m+1)/2}}
\\
&=   w_\gs {n^{(m+1)/2}}\int_0^1 \bex(t)^{m-1} \dd t
+ o\bigpar{n^{(m+1)/2}}.
  \end{split}
\end{equation}
This proves \eqref{t1a}--\eqref{t1w} in the case $\ell=1$, i.e., $\gs$ is a
block, and $m>1$. (The case $m=1$ is trivial.)

Consider now the general case when $\gs$ has $\ell\ge1$ blocks
$\gs^1,\dots,\gs^\ell$. 
We continue with the assumptions above; in particular $\pi=\pig$,
\eqref{dex} holds a.s., and the \PC{} hold for $\gam$.

Let $j_1,\dots,j_\ell$ be the positions in $\gs$
where the
blocks start; thus $j_1=1$ and $j_{p+1}=j_p+m_p$, $1\le p<\ell$.
Then $\bk=\bkk$ is an occurrence of $\gs$ in $\pi$ if and only if
each $(k_{j_p},\dots,k_{j_p+ m_p-1})$ is an occurrence of $\gs^p$ in $\pi$,
and furthermore $k_i<k_j$ whenever $i<k_p\le j$ for some $p$.
In particular, this implies, with the obvious definition of $\cA_k(\gs^p)$,  
\begin{equation}\label{nacka}
k_{j_1}<k_{j_2}<\dots<k_{j_\ell}
\quad\text{and}\quad
(k_{j_p},\dots,k_{j_{p}+m_p-1})\in \cA_{k_{j_p}}(\gs^p)
\text{ for all $p$}.
\end{equation}
On the other hand, if \eqref{nacka} holds and furthermore
$k_{j_{p+1}}>k_{j_p}+n^{0.6}$, say, for each $p<\ell$, 
then \eqref{skrapan} implies that,
assuming $n$ is large enough, 
$\bk$ is an occurrence of $\gs$ in $\pi$.
Consequently, with $q_p:=k_{j_p}$,
\begin{equation}
  \ns(\pig)
= \sum_{1\le q_1<\dots<q_\ell\le n} \prod_{p=1}^\ell |\cA_{q_p}(\gs^p)| 
+ o\bigpar{n^{(m+\ell)/2}}.
\end{equation}
We use again the intervals $I_\nu$ above, and obtain
\begin{equation}\label{agape}
  \begin{split}
      \ns(\pig)
&=\sum_{\nu_1<\dots<\nu_\ell}\prod_{p=1}^\ell
\Bigpar{\sum_{q_p\in I_{\nu_p}} |\cA_{q_p}(\gs^p)|} 
+ o\bigpar{n^{(m+\ell)/2}}.
  \end{split}
\end{equation}
For each $p$ with $m_p>1$, 
we argue as in \eqref{grig}--\eqref{gynt} 
and obtain
\begin{equation}\label{plikt}
  \begin{split}
\sum_{q_p\in I_{\nu}} |\cA_{q_p}(\gs^p)|
=w_{\gs^p} n^{(m_p-1)/2}\int_{I_{\nu}} \bex(x/n)^{m_p-1}\dd x
+ o\bigpar{n^{(m_p-1)/2}|I_\nu|}.
  \end{split}
\end{equation}
Furthermore, if $m_p=1$, 
then $\cA_k(\gs^p)=\set{k}$ and $|\cA_k(\gs^p)|=1$
for every $k\in[n]$, and thus \eqref{plikt} holds trivially,
with $w_\gs=1$ as given by \eqref{wgsblock}.

Finally, \eqref{agape}--\eqref{plikt} together with \eqref{wgs} yield,
with $W_\gs$ given by \eqref{t1w}, 
\begin{equation*}
  \begin{split}
      \ns(\pig)
&=
w_\gs n^{(m-\ell)/2}\sum_{\nu_1<\dots<\nu_\ell}\prod_{p=1}^\ell
\int_{I_{\nu_p}}\bex(x_p/n)^{m_p-1}\dd x_p
+ o\bigpar{n^{(m+\ell)/2}}
\\
&=
w_\gs n^{(m-\ell)/2}
\int_{0<x_1<\dots<x_\ell<n}\bex(x_1/n)^{m_1-1}\dotsm \bex(x_\ell/n)^{m_\ell-1}
\dd x_1\dotsm \dd x_\ell
\\&\hskip2em
+ o\bigpar{n^{(m+\ell)/2}}
\\
&=
n^{(m+\ell)/2}W_\gs + o\bigpar{n^{(m+\ell)/2}}.
  \end{split}
\end{equation*}
This completes the proof of \eqref{t1a}--\eqref{t1w}.

Since the proof shows \as{} convergence (under the coupling assumption in
the proof), joint convergence for several $\gs$ follows immediately.

In order to show moment convergence, we first prove another lemma.

For a Dyck path $\gamma\in\Dyckn$, let
\begin{equation}
  M(\gamma):=\max_{0\le i\le 2n}\gamma(i).
\end{equation}

\begin{lemma}
  \label{LM}
  \begin{thmenumerate}
  \item \label{LM1}
Let $M_n:=M(\gamma)$, where $\gamma$ is a uniformly random Dyck path of
length $2n$.
Then, for every fixed $r<\infty$, the random variables $(M_n/n\qq)^{r}$,
$n\ge1$, are uniformly integrable.
\item \label{LM2}
Let $\MXn:=\gDX(\pinwww)$.
Then, for every fixed $r<\infty$, the random variables $(\MXn/n\qq)^{r}$,
$n\ge1$, are uniformly integrable.
  \end{thmenumerate}
\end{lemma}

\begin{proof}
\pfitemref{LM1}
We use the well-known bijection between Dyck paths $\gam\in\Dyckn$ 
and ordered rooted trees $T_\gam$ with $n+1$ vertices, where $\gam$ encodes
the depth-first walk on $T_\gam$, see \eg{} 
\cite{AldousIII,Drmota}.
Then $M_n=\max \gam = H(T_\gam)$, the height of the tree $T_\gam$.
Furthermore, $T_\gam$ is a uniformly random ordered rooted tree with $n+1$
vertices, and can thus be represented as a conditioned Galton--Watson tree
with a Geometric offspring distribution, see 
\eg{} \cite{AldousI,Drmota};
hence we can apply \cite[Theorem 1.2]{SJ250}, and conclude that 
for all $n\ge1$ and $x\ge0$,
\begin{equation}
  \P\bigpar{M_n/\sqrt n \ge x}
= \P\bigpar{H(T_\gam)\ge x\sqrt n}
\le C e^{-c (x\sqrt n)^2/(n+1)}
\le C e^{-\cc x^2}.
\end{equation}
Consequently, for any fixed $r>0$, 
\begin{equation}
\E\bigpar{M_n/\sqrt  n}^{r+1}
=(r+1) \intoo x^r   \P\bigpar{M_n/\sqrt n \ge x}
\le C
\end{equation}
and the conclusion follows, 
see
\cite[Theorem 5.4.2]{Gut}.
 
\pfitemref{LM2}
By the BJS bijection, the uniformly random $\pinwww$ corresponds to a
uniformly random Dyck path $\gam\in\Dyckn$ by $\pinwww=\pi_\gam$.
We use the notation in \refSS{SSBJS}.

If $j\in\cD=\cD(\gam)$, then $j=D_i$ for some $i$, and thus, using \eqref{yi}, 
\begin{equation}\label{paxa}
0\le   \pi_\gam(j)-j=(A_i+1)-D_i=1+\gam(A_i+D_i)\le 1+M(\gam).
\end{equation}
On the other hand, if 
$j\notin\cD$, then
$D_i<j<D_{i+1}$ for some $i$, and by \cite[Lemmas 2.4 and 2.6]{HRS-I},
with very high probability $1-O(n^{-r-1})$, 
$|\pi_\gam(j)-j+y_i|<7n^{0.4}$ and thus (for large $n$)
\begin{equation}\label{paxb}
  |\pi_\gam(j)-j|\le 7n^{0.4}+y_i\le n^{0.5} + M(\gam).
\end{equation}
It follows from \eqref{paxa} and \eqref{paxb} that
with very high probability,
\begin{equation}\label{paxc}
  \MXn=\gDX(\pi_\gam)=\max_j|\pi_\gam(j)-j|\le n^{0.5}+M_n.
\end{equation}

Let $\cE_n$ be the event that \eqref{paxc} holds. Then
the exceptional event $\cE_n\cp$ has probability $O(n^{-r-1})$, say.
Consequently, using \eqref{paxc} on $\cE_n$ and the trivial bound
$\MXn\le n$ on $\cE_n\cp$,
and applying \ref{LM1},
\begin{equation}
  \E\bigpar{\MXn/\sqrt n}^{r+1} 
\le   \E\bigpar{1+M_n/\sqrt n}^{r+1} +O\bigpar{n^{(r+1)/2}\cdot n^{-r-1}} = O(1).
\end{equation}
The conclusion follows, see again \cite[Theorem 5.4.2]{Gut}.
\end{proof}

\begin{proof}[Competion of the proof of \refT{T1}]
  We have proved \eqref{t1a}--\eqref{t1w} above.
Furthermore, \eqref{nsbound} applied to each block $\gs^p$ shows that
\begin{equation}
  \ns(\pig)\le\prod_{p=1}^\ell n_{\gs^p}(\pig) = O\bigpar{n^\ell\gDX^{m-\ell}}.
\end{equation}
Hence, for any fixed $r>0$,
\begin{equation}\label{bra}
\bigpar{\ns(\pig)/n^{(m+\ell)/2}}^r
\le C n^{r\ell-r(m+\ell)/2}\gDX^{r(m-\ell)}
= C \bigpar{\gDX/n\qq}^{r(m-\ell)},
\end{equation}
which is uniformly integrable by \refL{LM}.
Consequently, 
the \lhs{} of \eqref{bra} is uniformly integrable, for any fixed $r>0$,
and the convergence in distribution \eqref{t1a} implies
convergence of moments too.
Convergence of mixed moments follows by the same argument.
\end{proof}

\section{Moment calculations}\label{Smoments}

Moments of the limiting random variable $W_\gs$ in \eqref{t1w}, and thus
asymptotics of the moments of $\ns(\pinwww)$, can often be calculated
explicitly. We do not know a single method that covers all cases, so we
present here some different methods, with overlapping applicability.
We give some example which
illustrate the methods, and leave further cases to the reader.

\subsection{Using known results}\label{SSknown}

In the special cases in Examples \ref{E1} and \ref{Elika}, 
$W_\gs$ is (up to the constant $w_\gs$) given by an integral
$\intoi\bex(t)^k$, or by a power of this integral. Hence moments of $W_\gs$
are given by moments of this integral, and these moments can be found
by recursion formulas, see \cite{SJ201} and the references there for $k=1$,
\cite{Nguyen} for $k=2$ and \cite{Richard} for the general case.

\begin{example}
For $\gs=\perm{231}$ and $\perm{312}$ we have by \refT{T1} and \eqref{w231}
the same limit in distribution
\begin{equation}\label{W231}
  W_{\perm{231}}=W_{\perm{312}}=\frac12\intoi\bex(t)^2\dd t.
\end{equation}
(In fact, $\npp{231}(\pinwww)$ and $\npp{312}(\pinwww)$ have the same
distribution for any $n$, as is easily seen because, in general,
$n_{\gs\qw}(\pi\qw)=\ns(\pi)$.)
By \cite[Table 2]{Nguyen}, \eqref{W231} yields \eg{} 
$\E W_{\perm{231}} = 1/4$,
$\E W_{\perm{231}}^2 = 19/240$
and
$\E W_{\perm{231}}^3 = 631/20160$.
\end{example}

\begin{example}
  \label{E1+1+1}
Let $\gs=\perm{214365}$. Thus $\gs$ consists of $\ell=3$ blocks, which all
are $\perm{21}$. Hence, \eqref{w21} yields $w_\gs=\wpp{21}^3=2^{-3/2}$,
and \eqref{elika} yields
\begin{equation}
 \begin{split}
  W_{\perm{214365}}
&
  =\frac{2^{-3/2}}6 \biggpar{\intoi\bex(t)\dd t}^3.    
  \end{split}
\end{equation}
Hence, using \eg{} \cite[Table 1]{SJ201},
$\E W_{\perm{214365}} = 5\sqrt\pi/512$.
\end{example}

\subsection{The joint density function}
First, for any $0<t_1<\dots<t_\ell$, the joint distribution of 
$\bigpar{\bex(t_1),\dots,\bex(t_\ell)}$ has an explicit density, see
\cite[Section 11.3, page 464]{RY}
(using the characterization of $\bex(t)$ as a three-dimensional Bessel bridge).
Thus, using \eqref{et1},  
$\E W_\gs$ can always be expressed
as a $2\ell$-dimensional multiple integral; furthermore,
higher moments can similarly be expressed using multiple integrals of
higher dimensions. 
However, we do not know how to calculate these integrals, except in the
simplest cases. 

In particular, this method works well for the expectation
in the special case
when there is only one non-trivial block (\ie, a block of length $>1$).
A special case of the joint density given in \cite[Section  XI.3]{RY}
is that 
for any fixed $t\in(0,1)$, $\bex(t)$
is a positive random variable with
the density
\begin{equation}
  \label{exdensity}
\frac{\sqrt2}{\sqrt{\pi t^3(1-t)^3}}
x^2e^{-x^2/(2t(1-t))},
\qquad x>0.
\end{equation}
(This also follows easily from \eqref{exbr}.)
Furthermore, \eqref{exdensity} implies  by a standard calculation which we omit
that if $t\in(0,1)$ and $r>-3$, then
\begin{equation}\label{xlex1}
\E [\bex(t)^r]
=2^{r/2+1} \pi\qqw \bigpar{t(1-t)}^{r/2} \gG\Bigparfrac{r+3}2.  
\end{equation}
We can now calculate $\E W_\gs$ for any $\gs$ that only has one non-trivial
block. 


\begin{example}\label{E1243}
  Let $\gs=\perm{1243}$, with blocks $\perm1$, $\perm1$, $\perm{21}$.
Thus $w_{\perm{1243}}=\wpp{1}^2\wpp{21}=\wpp{21}=2\qqw$ by \eqref{w21}.
Furthermore, by \eqref{t1w},
\begin{equation}
  W_{\perm{1243}}=w_{\perm{1243}}\int_{0<t_1<t_2<t_3<1}\bex(t_3)\dd t_1\dd t_2\dd t_3
=2^{-3/2}\intoi t^2\bex(t)\dd t.
\end{equation}
By \eqref{xlex1}, this yields
\begin{equation}
  \begin{split}
\E  W_{\perm{1243}}
&=2^{-3/2}\intoi t^2\E\bex(t)\dd t
=\pi\qqw\intoi t^{5/2} (1-t)^{1/2}\dd t
\\&
=\pi\qqw\frac{\gG(7/2)\gG(3/2)}{\gG(5)}
=\frac{5}{128}\sqrt\pi.
  \end{split}
\end{equation}
\end{example}


\subsection{Continuum random tree}
Our next method uses
a (minor) part of Aldous's theory of the Brownian continuum random tree
\cite{AldousI,AldousII,AldousIII}, 
in particular 
 \cite[Corollary 22 and Lemma 21]{AldousIII},  
which
among other things yield
a simple description (in terms of
binary trees with random edge lengths) of the distribution of the random
vector $\bigpar{\bex(U_1),\dots,\bex(U_\ell)}$, where $\ell\ge1$ and
$U_1,\dots,U_\ell\sim U(0,1)$ are \iid{} and independent of $\bex$.

In particular, this leads to the following.
(One can  obtain \eqref{lex1} also by integrating \eqref{xlex1}, but
the proof below requires less computations.)

\begin{lemma}\label{Lex}
  \begin{thmenumerate}
\item \label{Lex2}
If\/ $r>-2$, then
\begin{equation}\label{lex1}
\E\intoi \bex(t)^r\dd t
=2^{-r/2}  \gG\Bigpar{\frac{r}2+1}. 
\end{equation}
\item \label{Lexrs}
If\/ $r,s>-1$, then
\begin{multline}\label{lex2}
\E\intoi\intoi \bex(t)^r\bex(u)^s\dd t\dd u
\\
=
2^{-(r+s)/2}\Bigpar{\frac{r+s+2}{(r+1)(s+1)}-\frac{\gG(r+1)\gG(s+1)}{\gG(r+s+2)}}
\gG\Bigpar{\frac{r+s}2+1}.
\end{multline}
  \end{thmenumerate}
\end{lemma}

\begin{proof}
\pfitemref{Lex2}
For $\ell=1$, the description of \citet{AldousIII} simply says that
$2\bex(U_1)$ has a Rayleigh distribution with density $x e^{-x^2/2}$, $x>0$.
Hence, for any $r>0$,
\begin{equation*}
  2^r \E\intoi {\bex(t)^r}\dd t
= 2^r \E\bigsqpar{\bex(U_1)^r}
=\intoo x^{r+1}e^{-x^2/2}\dd x
= 2^{r/2}\gG\Bigpar{\frac{r}2+1},
\end{equation*}
where the final integral is evaluated using a standard change of variables,
see  \eg{}
\cite[(5.9.1)]{NIST}.
This yields \eqref{lex1}.

\pfitemref{Lexrs}
For $\ell=2$, the description of \citet{AldousIII} says that
\begin{equation}
\bigpar{2\bex(U_1),2\bex(U_2)}\eqd (L_1+L_2,L_1+L_3),  
\end{equation}
where
$(L_1,L_2,L_3)$ has the density
$(\ellx_1+\ellx_2+\ellx_3)e^{-(\ellx_1+\ellx_2+\ellx_3)^2/2}$, 
$\ellx_1,\ellx_2,\ellx_3>0$.
Consequently, for any $r,s>-1$,
using the change of variables $\wz=\ellx_1+\ellx_2+\ellx_3$,
$x=\ellx_2/\wz$, $y=(\ellx_1+\ellx_2)/\wz$,
\begin{equation*}
  \begin{split}
\hskip1em&\hskip-1em 
2^{r+s}\E \intoi\intoi \bex(x)^r\bex(y)^s\dd x\dd y
=
2^{r+s}\E\bigsqpar{\bex(U_1)^r\bex(U_2)^s}   
\\&
=\E\bigpar{(L_1+L_2)^r(L_1+L_3)^s}
\\&
=\int_{\ellx_1,\ellx_2,\ellx_3>0}(\ellx_1+\ellx_2)^r(\ellx_1+\ellx_3)^s(\ellx_1+\ellx_2+\ellx_3)
e^{-(\ellx_1+\ellx_2+\ellx_3)^2/2}\dd\ellx_1\dd\ellx_2\dd\ellx_3
\\&
=\int_{\wz=0}^\infty \iint_{0<x<y<1} y^r(1-x)^s \wz^{r+s+3}e^{-\wz^2/2}\dd x\dd y\dd \wz
\\&
=\frac{1}{r+1}\intoi \bigpar{1-x^{r+1}}(1-x)^s\dd x 
\cdot\intoo \wz^{r+s+3}e^{-\wz^2/2}\dd \wz
\\&
=\Bigpar{\frac{1}{(r+1)(s+1)}-\frac{\gG(r+2)\gG(s+1)}{(r+1)\gG(r+s+3)}}
2^{(r+s+2)/2}\gG\Bigpar{\frac{r+s}2+2}.
  \end{split}
\end{equation*}
Simple manipulations of the Gamma functions yield \eqref{lex2}.
\end{proof}

\begin{proof}[Proof of \refC{C1}]
Since $\gs$ is assumed to be indecomposable, \eqref{et1} holds with $\ell=1$
and $m_1=m$, and thus, using \refL{Lex},
\begin{equation}
  \E W_\gs = w_\gs\intoi \E [\bex(t)^{m-1}]\dd t
=w_\gs 2^{-(m-1)/2} \gG\bigpar{\tfrac{m+1}2}.
\end{equation}
Thus \eqref{c1} follows from \eqref{et1}.

Similarly, 
\begin{equation}
  \E [W_\gs^2] = w_\gs^2\E \intoi\intoi{\bex(t)^{m-1}\bex(u)^{m-1}}\dd
  t\dd u
\end{equation}
is given by \eqref{lex2}, and \eqref{c1var} follows.
\end{proof}

Note that in \eqref{t1w}, we integrate  only over $t_1<\dots<t_\ell$, while 
the method based on \cite{AldousIII} 
used here yields the integral
over $\oi^\ell$, without restriction on the order of the variables.
This was not a problem in \refC{C1}, 
when $\gs$ is indecomposable so  $\ell=1$.
The method also applies when $\ell>1$ in the special case 
when all blocks have the same lengths $m_1=\dots=m_\ell$;
see \refE{Elika}. In these cases,
higher moments of $W_\gs$ can be calculated by the same method, although the
calculations become more and more involved;
the method in \refSS{SSknown} seems simpler in these cases.

The method applies when $\ell=2$ 
for the expectation 
(but not for the variance or higher moments)
also when $m_1\neq m_2$, 
as consequence of the following lemma. 

\begin{lemma}\label{Lsymm}
If $r,s>-1$, then
\begin{equation}\label{lsymm}
\E  \int_{0<t_1<t_2<1}\bex(t_1)^r\bex(t_2)^s\dd t_1\dd t_2
=
\frac12 \E  \intoi\intoi\bex(t_1)^r\bex(t_2)^s\dd t_1\dd t_2.
\end{equation}
\end{lemma}

\begin{proof}
  Since the distribution of $\bex$ is invariant under reflection,
  $\bex(t)\eqd\bex(1-t)$ (as random functions),
\begin{equation}\label{lsymm2}
\E  \int_{0<t_1<t_2<1}\bex(t_1)^r\bex(t_2)^s\dd t_1\dd t_2
=
\E  \int_{1>t_1>t_2>0}\bex(t_1)^r\bex(t_2)^s\dd t_1\dd t_2
\end{equation}
and \eqref{lsymm} follows.
\end{proof}

\begin{proof}[Proof of \refC{C2}]
Lemmas \ref{Lsymm} and \refL{Lex}\ref{Lex2}  yield
\eqref{c2w}.
\end{proof}

As mentioned in \refS{S:intro}, Corollaries \ref{C1} and \ref{C2} proved
here imply \eqref{e12H}--\eqref{e231H}.

\subsection{Brownian bridge}

If all blocks of $\gs$ have odd length, then the exponents in \eqref{t1w}
are even, and thus we can use the representation \eqref{exbr}
and write 
$\E\bigsqpar{\bex(t_1)^{m_1-1}\dotsm \bex(t_\ell)^{m_\ell-1}}$
as the expectation of a polynomial in the jointly Gaussian variables 
$\brbr_k(t_i)$. This expectation is a polynomial in the covariances
$\Cov\bigpar{\brbr_k(t_i),\brbr_\gk(t_j)}=\gd_{k\gk}t_i(1-t_j)$ (for $t_i<t_j$), 
see \eg{} \cite[Thorem 1.28]{SJIII},
so
\eqref{et1} reduces to the integral of a polynomial over the given simplex,
which is calculated by elementary calculus.
Higher moments can be calculated similarly.

\begin{example}
Let $\gs=\perm{2314675}$, with the blocks $\perm{231}$, $\perm1$,
$\perm{231}$.
Then $w_\gs=w_{\perm{231}}^2=1/4$ by \eqref{w231}, and
\begin{equation}
  W_{\perm{2314675}}=\frac{1}4
 \int_{0<t_1<t_2<t_3<1} \bex(t_1)^2\bex(t_3)^2\dd t_1\dd t_2\dd t_3.
\end{equation}
Furthermore, using \eqref{exbr} and symmetry,
\begin{multline}
      \E\bigsqpar{\bex(t_1)^2\bex(t_3)^2}
=3 \E\bigsqpar{\brbr_1(t_1)^2\brbr_1(t_3)^2}
+ 6 \E\bigsqpar{\brbr_1(t_1)^2\brbr_2(t_3)^2}
\\
= 3\bigpar{t_1(1-t_1)t_3(1-t_3)+ 2 t_1^2(1-t_3)^2}
+ 6 t_1(1-t_1)t_3(1-t_3).  
\end{multline}
Hence, 
\begin{equation}
  \begin{split}
   &\E W_{\perm{2314675}}
\\
&\quad=
\frac{1}4\int_{0<t_1<t_3<1}\bigpar{9t_1(1-t_1)t_3(1-t_3)+ 6 t_1^2(1-t_3)^2}(t_3-t_1)
\dd t_1\dd t_3
\\&\quad
=
\frac{31}{3360}.
  \end{split}
  \raisetag{1\baselineskip}
\end{equation}
\end{example}

\newcommand\AAP{\emph{Adv. Appl. Probab.} }
\newcommand\JAP{\emph{J. Appl. Probab.} }
\newcommand\JAMS{\emph{J. \AMS} }
\newcommand\MAMS{\emph{Memoirs \AMS} }
\newcommand\PAMS{\emph{Proc. \AMS} }
\newcommand\TAMS{\emph{Trans. \AMS} }
\newcommand\AnnMS{\emph{Ann. Math. Statist.} }
\newcommand\AnnPr{\emph{Ann. Probab.} }
\newcommand\CPC{\emph{Combin. Probab. Comput.} }
\newcommand\JMAA{\emph{J. Math. Anal. Appl.} }
\newcommand\RSA{\emph{Random Struct. Alg.} }
\newcommand\ZW{\emph{Z. Wahrsch. Verw. Gebiete} }
\newcommand\DMTCS{\jour{Discr. Math. Theor. Comput. Sci.} }

\newcommand\AMS{Amer. Math. Soc.}
\newcommand\Springer{Springer-Verlag}
\newcommand\Wiley{Wiley}

\newcommand\vol{\textbf}
\newcommand\jour{\emph}
\newcommand\book{\emph}
\newcommand\inbook{\emph}
\def\no#1#2,{\unskip#2, no. #1,} 
\newcommand\toappear{\unskip, to appear}

\newcommand\arxiv[1]{\url{arXiv:#1.}}
\newcommand\arXiv{\arxiv}

\def\nobibitem#1\par{}

\end{document}